\documentclass[a4paper]{amsart}
\usepackage[top=2cm, bottom=2cm, left=2cm, right=2cm]{geometry}
\usepackage[utf8]{inputenc}
\usepackage{amssymb}
\usepackage{accents}
\usepackage[all]{xy}
\usepackage{amsmath, amsfonts, amsthm , amssymb, amscd}
\usepackage{pdflscape}
\usepackage{afterpage}
\usepackage{booktabs}
\usepackage{todonotes}
\usepackage{mathtools}
\usepackage{enumitem}
\usepackage{graphicx}
\usepackage{colortbl}
\graphicspath{ {./images/} }
\usepackage{xcolor}

\newtheorem*{theorem*}{Theorem A}
\newtheorem*{theorem**}{Theorem B}

\newtheorem{theorem}{Theorem}[section]

\newtheorem{corollary}[theorem]{Corollary}

\theoremstyle{definition}

\theoremstyle{remark}
\newtheorem{remark}[theorem]{Remark}

\newcommand{\RNum}[1]{\uppercase\expandafter{\romannumeral #1\relax}}

\newcommand{\Tor}{\ensuremath{\mathrm{Tor}}}

\newcommand{\ko}{\Bbbk}

\DeclareMathOperator{\Var}{Var}
\DeclareMathOperator{\Cov}{Cov}
\newcommand{\zk}{\mathcal Z_K}

\newcommand{\Pb}{\mathbb{P}}
\newcommand{\Zb}{\mathbb{Z}}
\newcommand{\FF}{\mathcal{F}}
\newcommand{\ZZ}{\mathcal{Z}}
\newcommand{\E}{\mathbb{E}}

\title{The Law of Large Numbers for the bigraded Betti numbers of the random moment-angle complex}
\author{Djordje Barali\'{c} \and Vlada Limic}
\address{\scriptsize{ Mathematical Institute SASA, Knez Mihajlova 36, p.p. 367, 11001, Belgrade, Serbia }}
\email{djbaralic@mi.sanu.ac.rs}
\address{\scriptsize{IRMA, UMR 7501 de l'Universit\'e de Strasbourg et du CNRS, 7 Rue Ren\'e Descartes, 67084, Strasbourg  Cedex, France}}
\email{vlada@math.unistra.fr}

\begin{document}

\maketitle

Let $(\Omega, \FF, \Pb)$ be a probability space. Let $n$ be a positive integer and $p\in [0,1]$.
 The {\em random simplicial $d$-complex} $Y_{n;p}^d$ on $n$ vertices and with parameter $p$, introduced by Linial and Meshulam in \cite{LiMes}, has the following probability law:\\
 - $Y_{n;p}^d$ takes values in the space of all simplicial complexes $K$ such that $\Delta_{d-1}^n \subset K \subset \Delta_d^n$\\
 - each possible $d$-simplex of $\Delta_d^n$ appears in $Y_{n;p}^d$ with probability $p$, independently.\\
 Recall that $\Delta_i^n$ denotes the $i$-skeleton of the standard $(n-1)$-simplex $\Delta^n$ on $n$ vertices.
 From now on we denote by $[n]$ the set of vertices, and by $Y^d(n,p)$ the law of $Y_{n;p}^d$.  
 The random simplicial $1$-complex is the well known {\em Erd\H{o}s-R\'enyi-Stepanov} random graph with law typically denoted by $G(n, p)$.
   
\textit{The moment-angle-complex} $\ZZ_K$ of a simplicial complex $K$ on $[n]$ is a topological space
$\ZZ_K :=\bigcup_{\sigma \in K} W(\sigma)$ where
$W(\sigma)=\prod_{i=1}^n V_i$, and where
$V_i=\begin{cases}
 D^2 & \mbox{if\,}  i\in\sigma\\
  S^1 & \mbox{if\,}  i\not\in\sigma,
\end{cases}$. Recall that $D^2=\{z\in \mathbb{C}| |z|\leq 1\}$ and $S^1=\{z \in \mathbb{C}| |z|=1\}$.
By definition, $X(\emptyset)=(S^1)^n$.

Cohomology ring of $\ZZ_K$ over $\ko$, where $\ko$ is a field or the ring of integers $\Zb$, was obtained by Buchstaber and Panov \cite{BuPan}, and Baskakov in \cite{Bask}. In particular, they proved that $H^{\ast, \ast} \left (\ZZ_K; \ko \right)\cong \Tor_{\ko [v_1, \dots, v_n]} (\ko[K],\ko)$, where $\ko [v_1, \dots, v_n]$ is the polynomial algebra with grading ${\rm deg}v_i=2$, and $\ko[K]$ is the {\em Stanley-Reisner ring} of $K$.
The latter object is the quotient graded ring
$\ko [v_1, \dots, v_n] / \mathcal{I}_K$, where $\mathcal{I}_K=\left(v_{i_1}\cdots v_{i_k} \left |\right. \{ i_1,\dots,i_k\} \not\in K\right)$. In words, $\mathcal{I}_K$  is the ideal generated by the square-free monomials corresponding to non-simplices of $K$.

Additive structure of $H^*(\zk;\ko)$ can thus be obtained using a well-known result from combinatorial commutative algebra, {\em{the Hochster's formula}}, which represents the above Tor-algebra as a direct sum of reduced simplicial cohomology groups of all full subcomplexes in $K$.
Hochster's formula states that
$$
\Tor^{-i,2j}_{\ko[v_{1},\ldots,v_{n}]}(\ko[K];\ko) \cong  \bigoplus_{J\subset [n], |J|=j}
\widetilde{H}_{|J|-i-1} (K_J;\ko),
$$
where $K_J$ is the \textit{full subcomplex}  of  $K$ on $J\subset [n]$.
Due to \cite{BuPan} this implies that $H^{l}\left (\mathcal{Z}_K; \ko \right ) \cong  \bigoplus_{J\subset [n]}\widetilde{H}_{l-|J|-1} (K_J;\ko)$.

\emph{The bigraded Betti numbers of $\ko (K)$} are the ranks of the bigraded components of Tor-groups 
$\beta^{-i, 2j} (\ko(K)):=\dim_{\ko} \Tor^{-i,2j}_{\ko[v_{1},\ldots,v_{m}]}(\ko[K];\ko)$. As a corollary we have
$\beta^{-i, 2j}(\ko(K))= \sum\limits_{J\subset [m], |J|=j} \dim_{\ko} \widetilde{H}_{j-i-1} (K_J;\ko)$.

In this work we introduce the {\em random moment-angle complex}
as $\ZZ_{Y_{n;p}^d}$, where $Y_{n;p}^d$ is as above.
We study the asymptotics of the ranks of the bigraded homology groups of $\ZZ_{Y_{n;p}^d}$, which due to the results given in the above paragraphs correspond to the bigraded Betti numbers
$\beta^{-i, 2j}(\ko(Y_{n;p}^d))$.

For fixed integers $i,j$
and $J\subset [n]$ such that $|J|=j$ 
let us denote
$$
 X_{J,i}^{(n,p)} \equiv X_{J,i}:= \dim_{\ko} \widetilde{H}_{|J|-i-1} (Y_{n;p}^d;\ko).
 $$
 Using the Hochster formula we get
\begin{equation}
\label{E:hoch concrete}
\beta^{-i, 2j}(\ko(Y_{n;p}^d)):= \sum_{J\subset [n], |J|=j} X_{J,i}.
\end{equation}
Due to dimensional reasons $\beta^{-i, 2j}(\ko(Y_{n;p}^d))=0$ for all $j\leq d$ and $j>n$.
In addition, $\beta^{-i, 2j}(\ko(Y_{n;p}^d))=0$ for
all $i,j$ such that $i\not\in\{j-d,j-d-1\}$ because $\widetilde{H}_{\ast}(Y_{n;p}^d;\ko)$ is trivial for $\ast\neq d-1, d$.

\begin{theorem}\label{t1} If $d+1\leq j\leq n$  then for each fixed $p\in (0,1)$ as $n\to \infty$
$$
\frac{\beta^{-j+d, 2j}(\ko(Y_{n;p}^d))}{{n \choose j}} \to f_j(p) \quad \text{and} \quad
\frac{\beta^{-j+d+1, 2j}(\ko(Y_{n;p}^d))}{{n \choose j}} \to g_j(p)
$$
almost surely, where $f_j$ and $g_j$ are both polynomials in $p$ of degree at most ${j \choose d+1}$. In fact, $f_j (p)=\E (\dim_{\ko} \widetilde{H}_{d-1} (Y_{j;p}^d;\ko)) $ and $g_j (p)=\E (\dim_{\ko} \widetilde{H}_{d} (Y_{j;p}^d;\ko))$.
\end{theorem}
\begin{proof} Let us observe that for all $J\subset [n]$, the random variables $X_{J,i}$ and $\beta^{-i, 2j}(\ko(Y_{j;p}^d)))$ have the same probability distribution. The Hochster formula implies that $$\E (\beta^{i, 2j}(\ko(Y_{n;p}^d)))={n \choose j} \E (\dim_{\ko} \widetilde{H}_{j-i-1} (Y_{j;p}^d;\ko)).$$ As $i, j$ are fixed, $\E (\dim_{\ko} \widetilde{H}_{j-i-1} (Y_{j;p}^d;\ko))$ is clearly a polynomial in $p$ with degree at most ${j \choose d+1}$. 

We present here only the proof for $i=j-d$, since the proof for $i=j-d-1$ is analogous. Due to Markov's or Chebyshev's inequality, for any $\varepsilon>0$ $$\Pb \left(\left| \frac{\beta^{-j+d, 2j}(\ko(Y_{n;p}^d))}{{n \choose j}}-f_j (p)\right|\geq \varepsilon\right)\leq \frac{\Var \left(\beta^{-j+d, 2j}(\ko(Y_{n;p}^d))\right)}{{{n \choose j}}^2 {\varepsilon}^2 }.$$
Again from \eqref{E:hoch concrete}, we have that $$\Var \left(\beta^{-j+d, 2j}(\ko(Y_{n;p}^d))\right)=\sum_{\substack{J\subset [n]\\ |J|=j}} \Var (X_{J, j-d})+\sum_{\substack{J_1, J_2\subset [n]\\ |J_1|=|J_2|=j\\ J_1\neq J_2}} \Cov (X_{J_1, j-d}, X_{J_2, j-d}). $$ As $\Var ( X_{J, j-d})=\Var (\dim_{\ko} \widetilde{H}_{d-1} (Y_{j;p}^d;\ko))$, we get $$\sum_{\substack{J\subset [n]\\ |J|=j}} \Var (X_{J, j-d})={n \choose j} \Var (\dim_{\ko} \widetilde{H}_{d-1} (Y_{j;p}^d;\ko))={n \choose j} a (p),$$ where $a (p)$ is a polynomial in $p$ of degree at most ${{j \choose d+1}}^2$.

Due to the definition of the random simplicial complex we have that $\Cov (X_{J_1, j-d}, X_{J_2, j-d})=0$ when $|J_1 \cap J_2|\leq d$. In fact, here $X_{J_1, j-d}$ are $X_{J_2, j-d}$ independent random variables. If $|J_1 \cap J_2|=m$ for $d+1\leq m \leq j-1$, then $\Cov (X_{J_1, j-d}, X_{J_2, j-d})=b_m (p)$, where $b_m (p)$ are again polynomials in $p$ of degree at most ${{j \choose d+1}}^2$. Let us denote by $A$, $B_{d+1}$, $\dots$, $B_{j-1}$ the maximal values of $a(p)$, $|b_{d+1} (p)|$, $\dots$, $|b_{j-1} (p)|$ on $[0, 1]$ respectively. Observe that those constants depend only on $d$ and $j$. Thus,
$$\Var \left(\beta^{-j+d, 2j}(\ko(Y_{n;p}^d))\right)\leq {n \choose j} A+\sum_{m=d+1}^{j-1} {n \choose m}{n-m \choose j-m} {n-j \choose j -m} B_m= O(n^{2j-d-1}).$$ 
Since $\sum\limits_n \frac{O(n^{2j-d-1})}{{{n\choose j}}^2} \approx \sum\limits_n 1/n^{d+1}<\infty$, the claim directly follows as an application of the Borel-Cantelli lemma.
\end{proof}

\vspace{-0.4cm}
\begin{remark}
For $\beta^{-1,2d+2}$ we also obtain an analogue of the central limit theorem.  
This is the special case when $j=d+1$, and for other $j$s as in the above theorem the second order asymptotic behavior is more complicated. 
We plan to return to this study in the near future. 
\end{remark}

\begin{corollary} For $d=1$, we write $Y^1_{n;p}=G_{n, p}$, and conclude that almost surely

\vspace{0.2cm}
$\beta^{-1, 6}(\ko(G_{n, p}))\slash n\to  p^3$, \ $\beta^{-2, 6}(\ko(G_{n, p}))\slash {n \choose 3}\to  (1-p)^2 (2+p)$,
$\beta^{-2, 8}(\ko(G_{n, p})\slash {n \choose 4} \to   2 p^3 (3p^3 -9 p^2-15 p+7)$.
\end{corollary}

\noindent
{\bf Acknowledgements.}
This research was supported in part by LabEx IRMIA Moyens de recherche grant 2019. The authors wish to thank Bernd Sturmfels for his support and encouragement. They met in 2018 at the Max-Planck Institute for Mathematics in the Sciences \textit{TAGS - Linking Topology to Algebraic Geometry and Statistics} workshop.

\end{document}